\documentclass[a4paper,10pt]{article}

\title{POLYCYCLIC, METABELIAN OR SOLUBLE  OF TYPE  (FP)$_{\infty}$ GROUPS WITH BOOLEAN  ALGEBRA OF RATIONAL SETS AND BIAUTOMATIC SOLUBLE GROUPS ARE VIRTUALLY ABELIAN }
\author{VITALY ROMAN'KOV\\ Omsk State Dostoevsky University and Omsk State Technical University,\\ 644077, Omsk, Russia, e-mail:romankov48@mail.ru}

\date{}

\usepackage{amsmath,amsthm,amsfonts}
\usepackage{graphicx}
\usepackage{hyperref}
\usepackage[usenames]{color}

\newtheorem{theorem}{Theorem}[section]
\newtheorem{lemma}[theorem]{Lemma}
\theoremstyle{definition}

\newtheorem{conjecture}[theorem]{Conjecture}
\newtheorem{definition}[theorem]{Definition}

\newcounter{comcount}

\begin{document}

\maketitle

\begin{abstract}
Let $G$ be  a polycyclic, metabelian  or soluble  of type (FP)$_{\infty}$  group  such that the class $Rat(G)$ of all rational subsets of $G$ is a boolean algebra. Then $G$  is virtually abelian. Every soluble biautomatic group is virtually abelian. 
\end{abstract}


\section{Introduction}
\label{se:intro}
The topic of this paper is two important concepts: rational sets and biautomatic structure. We study finitely generated soluble groups $G$  such that the class $Rat(G)$ of all rational subsets of $G$ is a boolean algebra. We conjecture that every such group is virtually abelian. Note that every finitely generated virtually abelian group satisfies to this property. We confirm this conjecture in the case where $G$ is a polycyclic, metabelian or soluble group of type FP$_{\infty}$. This conjecture remains  open in general case. It appeared that the notion  FP$_{\infty}$ helps to prove by the way that every soluble biautomatic group is virtually abelian. Thus, we give answer to known question posed in \cite{ECHLPT}.

We provide full proofs of four theorems attributed to Bazhenova, a former student of the author, stating that some natural assumptions imply that a soluble group  is virtually abelian.  The original proofs was given by her in collaboration with the author more than 14 years ago and had never been published. Now we fill this gap by presenting improved versions of these proofs. 

An excellent introduction to rational sets is \cite{Gil}, where the reader can find out basic definitions and fundamental results in this area.

Much of the basic theory of automatic and biautomatic groups is presented by Epstein and al. in \cite{ECHLPT}. One of the major open questions in group theory is whether or not an automatic group is necessarily biautomatic. The answer is not known even in the class of soluble groups. Note that  some known results about biautomatic groups remain open questions for automatic groups.Gersten and Short initiated in \cite{GS} the study of the subgroup structure of biautomatic groups. Among other results they established that a polycyclic subgroup of a biautomatic group is virtually abelian. Also they proved that if a linear group is biautomatic, then every soluble subgroup is (finitely generated) virtually abelian. 

\section{Preliminaries}
\label{se:prelim} 

Given an alphabet $\Sigma $,  recall that a {\it regular language} $R$ is a certain subset of the free monoid $\Sigma^{\ast}$ generated by $\Sigma ,$ which is empty or can be obtained by taking singleton subsets of $\Sigma $, and perform, in a finite number of steps, any of the three basic ({\it rational}) operations: taking union, string concatenation (product), and the Kleene star (generating of submonoid).

The construction of a set like $R$ is still possible when $\Sigma^{\ast}$ is changed by any monoid $M.$ Let $S$ be the set of all singleton subsets of $M.$ Consider the closure $Rat(M)$ of $S$ under the rational operations of union, product, and the formation of a submonoid of $M.$ In other words, $Rat(M)$ is the smallest subset of $M$ such that 

\begin{itemize}
\item $\emptyset \in Rat(M)$,
\item $A, B \in Rat(M) $ imply $A\cup B \in Rat(M),$
\item $A,B \in Rat(M)$ imply $AB \in Rat(M),$ where $AB =\{ab| a\in A, b \in B\}$,
\item $A \in  Rat(M)$ implies $A^{\ast} \in Rat(M),$ where $A^{\ast}$ is the submonoid of $M$ generated by $A.$ 
\end{itemize}

A {\it rational set} of $M$ is an element of $Rat(M).$ 

If  $\Sigma^{\ast}$ is a finite generated free monoid, then a rational set of $\Sigma^{\ast}$
is also called  {\it regular language}.

The rational sets of a monoid $M$ are precisely the subsets accepted by finite automata over $M.$ A {\it finite automaton} $\Gamma $ over $M$ is a finite directed graph with a distinguished initial vertex, some distinguished terminal vertices, and with edges labelled by elements from $M.$ The set {\it accepted} by $\Gamma $ is the collection of labels of paths from the initial vertex to a terminal vertex, where {\it label $\mu (p)$ of a  path} $p$ is  the product of labels of sequential edges in $p$ 

Recall two auxillary assertions. 

\begin{lemma}(The Pumping Lemma (see \cite{Gil}).
\label{le:1}
Let $M$ be a monoid, $R\in Rat(M).$ Then either $R$ is finite, or it contains a set of the form $aq^{\ast}b = \{aq^nb| n \geq 0\}, a,q,b \in M, q \not= 1.$ Moreover, if $M$ is a group, then the subset $aq^{\ast}b$ of $M$  can be written in the form $abb^{-1}q^{\ast}b = ab(b^{-1}qb)^{\ast}.$ So we might assume $b=1.$ 
\end{lemma}

\begin{lemma} (\cite{B}).
\label{le:2}
Let $G$ be a group, $H\leq G$ be a subgroup. If $R \subseteq H$ is  rational subset of $G$, then $R$ is a rational subset of $H.$   
\end{lemma}

As we have \ref{le:2}, we don't have to care if a set is rational as a subset of a larger group or of a smaller one. Note that with monoids which are not groups the situation may be different. 

\begin{lemma}(see \cite{Gil}).
\label{le:3}
Let $G$ be a group and $H \leq G.$ Then $H \in Rat(G)$ if and only if $H$ is finitely generated. If $R \in Rat(G)$ then gp$(R)\in Rat(G)$, and so finitely generated.
\end{lemma} 

\begin{lemma}(see \cite{B}).
\label{le:4}
Let $G$ be a group, $T \lhd G,$ and $\varphi : G \rightarrow G/T$ is the standard homomorphism. Then for every $R \in Rat(G)$ we have $\varphi (R) \in Rat(G/T).$ If $T$ is finitely generated then for every $S \in Rat(G/T)$ we have $\varphi^{-1}(S) \in Rat(G).$  
\end{lemma}

\section{Polycyclic groups}   

\bigskip
The  goal of this section   is to prove this:

\begin{theorem}
\label{th:pol}
Let $G$ be a polycyclic group. If  $Rat(G)$ is a boolean algebra, then $G$ is virtually abelian.
\end{theorem}

\begin{proof}
 Let us first consider a special case, that is the 'core' of the problem, in the sense concentrates all the difficult points of it.  So, let a group $G$ is a semidirect product $A 
\lambda H$, where $A \simeq \mathbb{Z}^r$ is a normal free abelian of rank $r$ subgroup of $G$, $H=$ gp$(h)$ is a cyclic subgroup of $G.$ Suppose that $Rat(G)$ is a boolean algebra. We are to prove that $G$ is virtually abelian. 

We may assume  $h$ to have infinite order.  First prove that some nontrivial element $g \in A$ and some exponent $h^m, m> 0,$ commute. Take an arbitrary nontrivial element $x \in A.$ Consider the set

$$R = \{h^{-n}xh^n | n \in \mathbb{Z}\}.$$

If $R$ is finite, then we have $h^{-n}xh^n =  h^{-t}xh^t$ for some $n > t$, which implies $[x, h^{n-t}] = 1$, hence we get what we need. Now assume that $R$ is infinite. Note that $R$ is rational, because it is equal to $(h^{-1})^{\ast}xh^{\ast} \cap A.$ Then by lemma \ref{le:1}, $R$ contains a subset $P = aq^{\ast}, a, q \in A, q \not= 1.$ Let $I$ be the set of indices such that $P = \{h^{-i}xh^i| i\in I\}$. Then $S= \{h^i| i \in I\} = h^{\ast} \cap (x^{-1}h^{\ast}P)$ is  infinite rational set. So it contains a subset $T = h^k(h^l)^{\ast}, k, l \in \mathbb{Z}, l > 0.$ The set $Q= \{f^{-1}xf| f \in T\} = (T^{-1}xT) \cap A$ is rational and subset of $P.$ We can assume that $k = 0.$ In other case we change 
$R, P, a$ and $q$ to $h^{-k}Rh^{k}, h^{-k}Ph^{k}, h^{-k}ah^{k}$ and $h^{-k}qh^{k},$ respectively. Now $T = (h^l)^{\ast}.$

Since $A$ is isomorphic to the free abelian group of rank $r$, we may regard it as a lattice of $\mathbb{R}^r.$ Let $| \cdot |$ be any standard norm on $\mathbb{R}^r.$ Since now we will use additive notation for the operation on $A$ as well as multiplicative one. 

Take an arbitrary real $\varepsilon > 0.$ Pick a positive integer $m$ such that $aq^m$ (or, in additive notation, $a + mq$ ) belongs to $Q$ and the inequality 

$$1/m|(h^{-l}ah^l) - a| < \varepsilon $$

\noindent
holds. This is possible, because $Q$ is infinite. Since $aq^m \in Q$, the element $h^{-l}(aq^m)h^l$ is in $Q$, so it can be written in the form $aq^p, p \in \mathbb{Z}, p \geq 0.$ Then 

$$h^{-l}ah^l + m(h^{-l}qh^l) = a + pq,$$
$$m(h^{-l}qh^l)- pq = a- h^{-l}ah^l,$$
$$
|h^{-l}qh^l- (p/m)q|= (1/m)|a- h^{-l}ah^l| < \varepsilon .$$  

It follows that $h^{-l}qh^l$ is a limit of elements of the form $sq, s \in \mathbb{R}, s \geq 0,$ so it has this form too. Clearly, $s$ is nonzero and rational, because $sq \in \mathbb{Z}^r \setminus \{0\}.$ Let $u$ be the greatest positive integer with the property: $q$ has the form $sq', q' \in A.$   Then $v = su$ is the greatest positive integer with the property: $sq$ has the form $vq', q' \in A.$ Since there exists 
 an automorphism of $A$ which takes $q$ to $sq,$ we get $u = sm$, so $s = 1.$ Then we have $h^{-l}qh^{l} = q,$ so we get what we need. 

Let us finish the proof. Let $f \in A, f \not= 1.$ Let $g, h^u$ ($g \not= 1, g \in A, u > 0$) commute. If $r,s$ are nonzero integers such that $g^r = f^s,$ then $(h^{-u}fh^u)^{rs}=
(h^{-u}f^sh^u)^{r} = g^{r^2}= f ^{rs}.$ As $A$ is abelian and torsion-free, $h^{-u}fh^u = f,$ so $f$ and $h^u$ commute. 

Now suppose that $g^r=f^s$ cannot hold unless $r = s =0,$ or, equivalently, $g^rf^s = g^nf^l$ cannot hold unless $r=n, s = l.$ Let $w = h^u.$ Consider the set 

$$R = (wg)^{\ast}f^{\ast} \cap w^{\ast}(gf)^{\ast} = \{w^ng^nf^n | n \geq 0\}.$$ 

Let

$$S = R(g^{\ast}\cup (g^{-1})^{\ast}) \cap w^{\ast}f^{\ast} =$$
$$\{w^nf^n| n \geq 0\}.$$

Take a subset $aq^{\ast} \subseteq S, q \not= 1.$ Let $a = w^nf^n, aq = w^mf^m, n \not= m.$ Then

$$aq^2 =(aq)a^{-1}(aq) = w^{2m-n}(w^{n-m}f^{m-n}w^{m-n})f^m$$

\noindent
also has the form $w^tf^t, t \geq 0.$ Then $t = 2m-n.$ Hence

$$f^{2m-n} = f^m(w^{n-m}f^{m-n}w^{m-n}),$$

\noindent
so

$$f^{m-n}= w^{n-m}f^{m-n}w^{m-n},$$

\noindent
then

$$f = w^{n-m}fw^{m-n}.$$

Then $f$ and $h^{m-n}$ commute. 

We proved  that for each $f \in A$ there exists $t \in \mathbb{Z}, t > 0,$ such that $f$ and $h^t$ commute.  Let $f_1, ..., f_r$ be free generators for $A.$ Pick some nonzero $t_1, ..., t_r$ such that for each $i$ the elements $f_i$ and $h^{t_i}$ commute. Let $N =  t_1 ... t_r.$ Then $h^N$ and every $f_i$ commute. Hence the group $M$ generated by $A$ and $h^N$ is abelian. Clearly, $M$ has finite index in $G.$ Thus $G$ is virtually abelian. 

Now we are ready to prove the theorem by using induction on solubility length. If $G$ is abelian the statement is trivial. Let $G$ be nonabelian.The derived subgroup  $G'$ has smaller solubility length, and, as it is finitely generated, its rational subsets are a boolean algebra.  Then $G'$ is virtually abelian. Let $G = $ gp$(G', g_1, ..., g_j)$. Consider the series of subgroups $G' = G_0 \leq G_1 \leq G_j = G$, where $G_{i+1} = $ gp$(G_i, g_{i+1})$ for $i = 0, ..., j-1.$ Clearly, all $G_i$ are normal in $G.$ Prove by induction on $i$ that all $G_i$ are virtually abelian. Suppose that $H_i$ is a finite index normal abelian subgroup of $G_i.$  We have proved above that gp$(H_i, g_{i+1}$ is virtually abelian. Hence $H_{i+1}$ is virtually abelian. 

\end{proof}

\section{Metabelian groups}   

\bigskip
 Recall that a group $G$ is said to have the {\it Howson property} (or to be a {\it Howson group}) if the intersection $H \cap K$ of any two finitely generated subgroups $H, K$ of $G$ is finitely generated subgroup. Let  $G$ be a group in which   $Rat(G)$ is a boolean algebra. Then  $G$ has the Howson property. Indeed, a subgroup of arbitrary group is a rational set if and only if it is finitely generated.  By our assumption the intersection $H\cap K$ is  rational set. Hence, $H \cap K$ is finitely generated subgroup. 

All finitely generated metabelian nonpolycyclic Howson groups are characterized as follows.

\begin{theorem}( Kirkinskij  \cite{Kir}).
\label{th:12}
Let  $G$ be a finitely generated metabelian nonpolycyclic group.  Then the following properies are equivalent: 
\begin{enumerate}
\item $G$ has the Howson property,
\item the finitely generated nonpolycyclic subgroups of $G$ have finite indexes,
\item $G$  has a subgroup $H$ of finite index containing a normal finite subgroup $T$ such that $H/T \simeq $ gp$(x, a| [a,a^{x^i}] =1, i \in \mathbb{Z}, a^{f(x)} =1)$ with $f(x)$ being irreducible over $\mathbb{Z}$ polynomial with integral coefficients such that deg$f(x)\geq 1 $ and for every $n \in \mathbb{N}$ this polynomial does not divide any polynomial in $x^n$ of degree deg$f(x)-1.$ If $f(x) = q_0x^m + q_1x^{m-1} + ... + q_m,$ then $a^{f(x)}$ means $(a^{x^m})^{q_0}(a^{x^{m-1}})^{q_1} ... a^{q_m}.$
\end{enumerate}
\end{theorem} 

\begin{theorem}
\label{th:13}
Let   $G$  be a finitely generated metabelian group such that  $Rat(G)$  is a boolean algebra. Then  $G$ is virtually abelian group. 
\end{theorem}

\begin{proof}
If $G$ is polycyclic the statement follows by theorem \ref{th:pol}. 
Suppose  $G$ is not polycyclic. Then by theorem  \ref{th:12}  $G$ has a series  $1 \leq T \lhd H \leq G$. Since   $H$  is finitely generated, then  $Rat(H)$ is a boolean algebra. Since  $T$  is finite then by lemma \ref{le:4} $Rat(H/T)$  is a boolean algebra.

Let by theorem \ref{th:12} $H/T \simeq $ gp$(x, a| [a,a^{x^i}] =1, i \in \mathbb{Z}, a^{f(x)} =1),$ $f(x) =  q_0x^m + q_1x^{m-1} + ... + q_m.$
Note that every element   $g \in H/T$ can be expressed as  $g = x^ka^{\frac{r(x)}{x^l}},$  where  $k, l \in \mathbb{Z}, l \geq 0,$ and  $r(x)$  is a polynomial  with integer coefficients. One has   $g = 1$ if and only if  $k = 0$ and $r(x)$   divides to $f(x)$  in the polynomial ring   $\mathbb{Z}[x].$

Fix some numbers $p, d \in \mathbb{Z}, p, d > 0.$ Define the following rational sets in $H/T$: 
$$
R_1 = ((a^dx^p)^{-1})^{\ast}(\{a^d, a^{d+1}\}x^p)^{\ast},$$
$$R_2 = (x^{-p})^{\ast}(\{1, a\}x^p)^{\ast},$$
$$R_3 = ((a^dx^{-p})^{-1})^{\ast}(\{a^d, a^{d+1}\}x^{-p})^{\ast},$$
$$
\label{eq:48}
R_4 = (x^p)^{\ast}(\{1, a\}x^{-p})^{\ast}.
$$

By our assumption all  intersections $R_i\cap R_j$ for $i,j = 1, ..., 4,$  are rational. 

Any element of   $R_1$  can be written in the form: 

$$
(a^dx^p)^{-l+k}(a^{\epsilon_1})^{x^{kp}}(a^{\epsilon_2})^{x^{(k-1)p}}...(a^{\epsilon_k})^{x^{p}},
$$
 
\noindent
where   $l, k \in \mathbb{Z}, l, k \geq 0,$  and    $\epsilon_i = 0$     or  $\epsilon_i = 1.$ 

Any element of  $R_2$ can be written in the form:   

$$
x^{(-l+k)p}(a^{\epsilon_1})^{x^{kp}}(a^{\epsilon_2})^{x^{(k-1)p}}...(a^{\epsilon_k})^{x^{p}},
$$

\noindent
$l, k \in \mathbb{Z}, l, k \geq 0,$ $\epsilon_i = 0$   or  $\epsilon_i = 1.$  

Any element of    $R_3$ can be written in the form: 

$$
(a^dx^{-p})^{-l+k}(a^{\epsilon_k})^{x^{-p}}(a^{\epsilon_{k-1}})^{x^{-2p}} ... (a^{\epsilon_1})^{x^{-kp}},
$$ 

\noindent
$l, k \in \mathbb{Z}, l, k \geq 0,$  where   $\epsilon_i = 0$  or  $\epsilon_i = 1.$   

Any element of  $R_4$ can be written in the form:  

$$
x^{(l-k)p}(a^{\epsilon_k})^{x^{-p}}(a^{\epsilon_{k-1}})^{x^{-2p}} ... (a^{\epsilon_1})^{x^{-kp}},
$$ 

\noindent
$l, k \in \mathbb{Z}, l, k \geq 0,$ $\epsilon_i = 0$   or $\epsilon_i = 1.$   Also note that for    $n> 0$ we have:

$$
(a^dx^p)^n = x^{np}(a^{d})^{x^{np}} ... (a^d)^{x^p},
$$

$$
(a^dx^p)^{-n} = x^{-np}(a^{-d})^{x^{-(n-1)p}} ... (a^{-d})^{x^{-p}}a^{-d},
$$

$$
(a^dx^{-p})^n = x^{-np}(a^{d})^{x^{-np}} ... (a^d)^{x^{-p}},
$$

\noindent
and

$$
(a^dx^{-p})^{-n} = x^{np}(a^{-d})^{x^{(n-1)p}} ... (a^{-d})^{x^p}a^{-d}.
$$

The sets $R_1$ and $R_2$ contain the elements

$$a^{x^{kp}} = ((a^dx^p)^{-1})^k\cdot (a^{d+1}x^p)(a^dx^p)^{k-1}=$$
$$(x^{-p})^k\cdot (ax^p)(x^p)^{k-1}, k = 1, 2, ... .$$

Similarly, the sets $R_3$ and $R_4$ contain the elements $a^{x^{-kp}}, k = 1, 2, ... .$

Let  $N = $ ncl$(a)$ be the normal closure of the element   $a$ in $H/T$ (that is the minimal normal subgroup of $N/T,$ containing  $a$).  If the sets   $S_1 = R_1 \cap R_2$  and   $S_2 = R_3 \cap R_4$ lie in  $N,$ then the subgroup   $M$ generated by  $S_1 \cup S_2 \cup \{a\},$ is the normal closure of   $a$ in the subgroup generated by  $a$  and $x^p.$  By lemma \ref{le:3} every subgroup generated by a rational set is rational and finitely generated.  Since   $N$ is generated by a finite set of subgroups that are conjugate to  $M,$ it is finitely generated too. In the case   $H/T$ and $G$ are polycyclic. We get contradiction to our assumption. Hence at least one of the subsets   $S_i, i = 1, 2,$ does not lie in  $N.$ Then one of the following equalities is true:  

$$
(a^{\epsilon_n})^{x^{np}} ... 
(a^{\epsilon_{k+1}})^{x^{(k+1)p}}(a^{d+\epsilon_k})^{x^{kp}}\cdot 
$$
\begin{equation}
\label{eq:57}
(a^{d+\epsilon_{k-1}})^{x^{(k-1)p}} ... (a^{d+\epsilon_1})^{x^{p}} = 1,
\end{equation}

\begin{equation}
\label{eq:58}(a^{\epsilon_n})^{x^{np}}...(a^{\epsilon_{1}})^{x^{p}}(a^{-d})(a^{-d})^{x^{-p}}...(a^{-d})^{x^{-1p}} = 1,
\end{equation}

\begin{equation}
\label{eq:59}
(a^{d+\epsilon_{1}})^{x^{-p}} ... (a^{d+\epsilon_k})^{x^{-kp}}
(a^{\epsilon_{k+1}})^{x^{-(k+1)p}}   ... (a^{\epsilon_n})^{x^{-np}} = 1,
\end{equation}

\begin{equation}
\label{eq:60}
(a^{-d})^{x^{lp}} ...
(a^{-d})^{x^{p}}(a^{-d})(a^{\epsilon_1})^{x^{-p}}(a^{\epsilon_k})^{x^{-kp}} = 1,
\end{equation}

\noindent
$l, k, n \in \mathbb{Z}, l \geq 0, n \geq k > 0, \epsilon_i \in \{-1, 0, 1\}.$    If  the absolute value $\mu $ of  one of the coefficients  $q_0, q_m$ in $f(x)$ is greater than $3$, we may assume that chosen number  $d$ is such that  $d-1, d$ and $d+1$  do not divide to  $\mu .$  It follows that all equalities  \ref{eq:57} --    \ref{eq:60} failed.  Further, both of the coefficients  $q_0, q_m$   cannot be $\pm 1,$  because in the case   $H/T$ is polycyclic. Thus  $q_0$ or  $q_m$  is equal to $2$ or $3.$  We set  $d = 2^2\cdot 3^2 + 2 = 38.$ Then
 $d-1, d, $  and $d+1$  do not divide to   $\mu^2.$   We can assume that  $p > m+1.$ Then everyone of the equalities   \ref{eq:57} -- \ref{eq:60} implies that all the coefficients of   $f(x)$  divide to  $\mu ,$ and   $|q_0| = |q_m| = \mu .$ Then the normal closure  $K=$  ncl$(a^{\mu }) \lhd H/T$  is finitely generated. By lemma  \ref{le:4}  $Rat((H/T)/K)$  is a boolean algebra.   The quotient  $(H/T)/K$ is a homomorphic image of the wreath product $Z_{\mu } wr  Z,$ where $\mu $ is prime. Hence, either $(H/T)/K$ is polycyclic, or is $Z_{\mu } wr  Z.$ In the first case $(H/T)/K$ satisfies to the ascending chain condition, and since    $K$ is finitely generated abelian, then  $H/T$ satisfies to the ascending chain condition too. Every soluble group with the ascending chain condition is polycyclic (see for instance \cite{KM}). Hence  $H/T$ is polycyclic. The second case is impossible, because    $Z_{\mu } wr  Z$ is not Howson group  (see \cite{Kir}). 

\end{proof}

\section{  Solvable  groups of type $FP_{\infty}$   with boolean algebras of rational subsets.}

\bigskip

\begin{definition}
\label{def:7}
A group  $G$ is said to be of type   FP$_{\infty}$ if and only if there is a projective resolution  
\begin{equation}
\label{eq:61} ... \rightarrow P_j \rightarrow ... \rightarrow P_2 \rightarrow P_1 \rightarrow P_0 \rightarrow \mathbb{Z} \rightarrow 0
\end{equation}

\noindent
of finite type: that is, in which every $P_j$ is finitely generated. 
\end{definition}

\begin{definition}
\label{def:8}
A group $G$ has finite cohomological dimension if and only if there is a projective resolution  
\begin{equation}
\label{eq:62} ... \rightarrow P_n \rightarrow ... \rightarrow P_2 \rightarrow P_1 \rightarrow P_0 \rightarrow \mathbb{Z} \rightarrow 0
\end{equation}

\noindent
of finite length: that is, in which the $P_i$ are zero from some point on. 
\end{definition}

The following remarkable theorem is a base in our proof of the main result of this section. 

\begin{theorem} (Kropholler  \cite{K3},  see also \cite{K1}).
\label{th:14}
If   $G$ is a soluble group of type   FP$_{\infty}$  then vcd$(G) < \infty $.   
\end{theorem}

Also we need in a standard statement as follows.

\begin{lemma}
\label{le:14}
If   cd$(G) < \infty $, then:
1) $G$ is torsion-free,\\
2) there is  $n > 0,$  such that: if  $A \leq G, A \simeq \mathbb{Z}^k,$ then   $k\leq n.$
\end{lemma}

At this section we will specialize to soluble groups of type FP$_{\infty}$ and establish: 

\begin{theorem} 
\label{th:15}
If  $G$   is a finitely generated soluble group of type (FP)$_{\infty},$  such that $Rat(G)$ is a boolean algebra. Then    $G$ is virtually abelian group.
\end{theorem}

\begin{proof}
By theorem   \ref{th:14} there is a subgroup of finite index $H \leq G$, that has finite cohomological dimension. The subgroup  $H$ is finitely generated and thus rational in  $G.$ Then $Rat(H)$ is a boolean algebra. By  lemma \ref{le:14}  $H$ is torsion-free. Hence, every finitely generated abelian subgroup of  $H$ is a free abelian of bounded rank. By Kargapolov's theorem  (see \cite{KM}  or  \cite{LR}), $H$ has finite rank, i.e.,  there is a finite number $r$ such that every finitely generated subgroup of $H$ can be generated by   $r$ elements and $r$ is least such integer  ({\it Pr\"{u}fer} or {\it Mal'cev} rank). 

Thus,   $H$ is a soluble torsion-free group of finite rank. By Robinson-Za\v{i}cev theorem    (see  \cite{LR}), every finitely generated soluble group with finite rank is a minimax group.  It means there is a subnormal series 

\begin{equation}
\label{eq:63}
1 = K_0 \leq K_1 \leq ... \leq K_n = K, 
\end{equation}

\noindent
in which every quotient $K_{i+1}/K_i$ satisfies to the ascending or descending condition.   Also we know (see \cite{LR}), that every soluble torsion-free minimax group is nilpotent-by-(virtually abelian).  

Let $N$  be a nilpotent normal subgroup of  $H,$  such that   $H/N$ is virtually abelian.   We will prove that  $N$  should be abelian. Let $1 \leq \zeta_1(N) \leq \zeta_2(N) \leq ... \leq \zeta_j(N) = N$ be the upper central series of $N.$  If $u \in N, y \in \zeta_2(N)\setminus \zeta_1(N)$ do not commute, then $[u, y] \in \zeta_1(H)$ has infinite order.  We know that the quotient of torsion-free nilpotent group by the center is torsion-free (see  \cite{KM}). Then the subgroup gp$(g,y)$ is isomorphic to the free nilpotent of rank $2$ and class $2$ group (UT$_3(\mathbb{Z})$ or the {\it Heizenbergh} group). But UT$_3(\mathbb{Z})$ is obviously non virtually abelian, hence $Rat(UT_3(\mathbb{Z}))$ is not boolean algebra by \cite{B}. It follows that  $\zeta_2(N) = \zeta_1(N),$ thus $N$  is abelian. Then  $H$ is extension of the abelian normal subgroup   $N$ with virtually abelian group  $H/N.$ Then by theorem \ref{th:13} $G$ is virtually abelian. 

\end{proof}

\section{Solvable biautomatic groups.}

\bigskip
The class of automatic groups is one of the main classes studying by geometric group theory. Different properties of automatic and biautomatic groups are described in the classical monography  \cite{ECHLPT}. See also \cite{GS}.  We give the main definitions. 

Let  $(A, \lambda , L)$ be a rational structure for $G.$ Recall, that  $A$ is a finite alphabet,  $\lambda $ is a homomorphism of the free monoid   $A^{\ast}$ onto  $G,$ $L$ is a regular language in  $A^{\ast}$ such that  $\lambda (L) = G.$  The set  $A$  is  a generating set for  $G,$ considered as a monoid. We assume that  $A$ is symmetric: that is $A$ contains with every its element   $a$ its formal inverse  $a^{-1}.$ We assume that homomorphisms of $A^{\ast}$ to groups map formally inverse elements to inverse images. The set  $L$ is considered as the set of normal forms of expressions of elements of $G.$ We add to   $A$ a new symbol $\$.$ Consider alphabet   $A^{2\$},$  consisting of the pairs  $(b, c),$  where   $b, c \in A \cup \$.$ Take the corresponding free monoid $(A^{2\$})^{\ast}.$ The homomorphism   $\lambda $ is naturally extended to homomorphism of the free monoid  $(A^{2\$})^{\ast}$ onto $G^2.$ One has  $\lambda (\$) = 1.$

A rational structure  $(A, \lambda , L)$ is called  {\it automatic}  for $G$ under the following conditions are satisfied:   

\bigskip
\begin{equation}
\label{eq:}
\{(u, v) \in L^{2\$}: \lambda (u) = \lambda (v)\}
\end{equation}

\noindent
and for every    $a \in A$ the language  

\begin{equation}
\label{eq:}
\{(u, v) \in L^{2\$} : \lambda (u)  = \lambda (va)\}
\end{equation} 
\noindent
is regular in  $A^{2\$}.$

\bigskip
\begin{definition}
\label{def:9}
\
A group  $G$  is said to be  {\it automatic} if $G$ has an automatic structure  $(A, \lambda , L).$  
\end{definition}

An automatic structure is said to be  {\it biautomatic}, if for every   $a \in A$ the language  

\begin{equation}
\label{eq:66} 
\{(u, v) \in L^{2\$} : \lambda (u)  = \lambda (av)\}
\end{equation} 
\noindent
is regular in  $A^{2\$}.$

\bigskip
\begin{definition}
\label{def:10}
A group    $G$ is said to be  {\it biautomatic} if $G$ has a biautomatic structure  $(A, \lambda , L).$
\end{definition}

In \cite{ECHLPT} a question is formulated: is every biautomatic group virtually abelian?  We give a positive answer to this question by the following theorem. This result has been obtained by Bazhenova,  Noskov,  Remeslennikov and the author with using of information received them from  Kropholler. 

\begin{theorem}(Bazhenova, Noskov, Remeslennikov, Roman'kov). 
\label{th:16}

\noindent
Let  $G$   be a finitely generated soluble biautomatic group. Then    $G$ is virtually abelian. 
\end{theorem}

\begin{proof}
By \cite{ECHLPT}, theorem 10.2.6,  every soluble biautomatic group has type FP$_{\infty}.$ Hence by theorem \ref{th:14}  $G$ has a subgroup of finite index  $H$  with cd$(H) < \infty .$ By lemma \ref{le:14}  $H$ is torsion-free. Moreover, all abelian subgroups of $H$ has  bounded rank. Hence by Kargapolov's theorem (see \cite{KM}  or  \cite{LR})  $H$ has finite rank. By  Robinson-Za\v{i}cev theorem    (see  \cite{LR}), every finitely generated soluble group with finite rank is a minimax group.   Also we know (see \cite{LR}), that every soluble torsion-free minimax is nilpotent-by-(virtually abelian). Let  $N$  be a nilpotent normal subgroup of   $H$  such that  $H/N$ is virtually abelian.  

We will prove that   $H$ is abelian. Let $1 \leq \zeta_1(N) \leq \zeta_2(N) \leq ... $ be the upper central series in $N.$  Let $u \in N, y \in \zeta_2(N)\setminus \zeta_1(N)$ don't commute. Then $[u, y] \in \zeta_1(H)$ is a nontrivial element of infinite order. 
 Then  the subgroup gp$(g,y)$ is isomorphic to the free nilpotent of rank $2$ and class $2$ group UT$_3(\mathbb{Z}).$   It cannot happen since UT$_3(\mathbb{Z})$ is polycyclic but not virtually abelian. Indeed, by \cite{GS} every polycyclic subgroup of biautomatic group is virtually abelian.  
Thus, $\zeta_2(H) = \zeta_1(H), $ and $H =\zeta_1(H)$ is abelian.   

The group  $H$ is finitely generated and virtually metabelian.  Every subgroup of a finite index in a biautomatic group is biautomatic  \cite{GS}. Thus,  $H$ is biautomatic. Then $H$ satisfies to the minimal condition for centralizers  (see \cite{R} or \cite{LR}).  It means: if   
 ${\cal\bf  S}(H)$ be a class of subgroups of  $H$ of type  $M = C_H(X) = \{h \in H: \forall x \in X \  [h,x] = 1\},$  where  $X \subseteq H,$ then descending sequence   $M_1 \geq M_2 \geq ... \geq M_l ... $ of subgroups in  ${\cal\bf S}(H)$ stabilizes on a finite step. Then there is a number   $l$ such that   $M_l = M_{l+1} = ... .$ By \cite{GS} each biautomatic group with this property satisfies to the maximal condition  on abelian subgroups. By Mal'cev's theorem  (see \cite{KM}) a soluble group with this condition is polycyclic. Hence   $H$ is polycyclic. Then by  \cite{GS} $H$ is virtually abelian. Hence $G$ is virtually abelian. 
  
\end{proof}

\bigskip
{\large Acknowledgment.} The author was supported in part by Russian Research Fund, grant №14-11-00085.

\end{document}